\documentclass[fleqn,preprint,3p,a4paper]{elsarticle}
\usepackage{amssymb}
\usepackage{amsmath}
\usepackage{amsthm}
\usepackage{dcolumn}
\usepackage{endnotes}
\usepackage{tabularx}
\usepackage[matrix,arrow]{xy}
\usepackage{wasysym}
\usepackage{graphicx}

\newtheorem{theorem}{Theorem}[section]
\newtheorem{proposition}[theorem]{Proposition}
\newtheorem{lemma}[theorem]{Lemma}
\newtheorem{corollary}[theorem]{Corollary}

\theoremstyle{definition}
\newtheorem{definition}[theorem]{Definition}
\newtheorem{example}[theorem]{Example}
\newtheorem{remark}[theorem]{Remark}
\newtheorem{question}[theorem]{Question}

\newcommand{\ir}{{\mathsf{Irr}}}

\newcommand{\cl}{{\rm cl}}
\newcommand{\ii}{{\rm int}}

\newcommand{\ua}{{\uparrow}}
\newcommand{\da}{{\downarrow}}

\begin{document}

\begin{frontmatter}

\title{On strong $R$-spaces\tnoteref{t1}}
\tnotetext[t1]{This research was supported by the National Natural Science Foundation of China (Nos. 12471070, 12071199, 12261040), the second author was supported by the Project of Suqian Sci\&Tech Program (Grant No. K202441), the Foundation of PhD start-up of Suqian University (Nos. 2024XRC003).}

\author[X. Wen]{Xinpeng Wen}
\ead{wenxinpeng2009@163.com}
\address[X. Wen]{College of Mathematics and Information, Nanchang Hangkong University, Jiangxi 330063, China}

\author[M. Bao]{Meng Bao}
\ead{mengbao95213@163.com}
\address[M. Bao]{School of Mathematics, Suqian University, Suqian 223800, China}

\author[X. Xu]{Xiaoquan Xu\corref{mycorrespondingauthor}}
\cortext[mycorrespondingauthor]{Corresponding author}
\ead{xiqxu2002@163.com}
\address[X. Xu]{School of Mathematics and Statistics,
Minnan Normal University, Zhangzhou 363000, China}

\begin{abstract}
In this paper, we mainly investigate some basic properties of strong $R$-spaces. It is shown that the property of being a strong $R$-space is closed-hereditary, saturated-hereditary and retractive, but not finite productive. Hence the category $\mathbf{S}$-$\mathbf{Top}_r$ of strong $R$-spaces and continuous mappings is not reflective in  the category $\mathbf{Top}_0$ of $T_0$-spaces and continuous mappings. It is proved that a $T_0$-space $(X, \tau)$ is a strong $R$-space iff every nonempty $\tau$-closed subset of $X$ is compact in $(X, \tau^{d})$, where $\tau^d$ is the de Groot dual of $\tau$; consequently, if $(X, \tau)$  is a strong $R$-space (especially, if $(X, \tau)$ is a coherent well-filtered space), then $\tau \subseteq \tau^{dd}$. Therefore, for any locally compact strong $R$-space $(X, \tau)$, we have $\tau=\tau^{dd}$. Finally, we investigate conditions under which the Smyth power space and Scott power space of a $T_0$-space is a strong $R$-space. Several such conditions are given.
\end{abstract}


\end{frontmatter}

\section{Introduction}

In domain theory and non-Hausdorff topology, the $d$-spaces, well-filtered spaces and sober spaces form three of the most important classes of $T_0$-spaces (cf. \cite{Abramsky-Jung-1994, GHKLMS-2003, Goubault-2013, Jia-Jung-2016, Xu-Zhao-2021}). In order to uncover more finer links between $d$-spaces and $T_2$-spaces, the $R$-spaces, strongly well-filtered spaces and strong $d$-spaces were introduced and investigated in \cite{Lawson-Xu-2024-1, Xu-2026, Xu-Zhao-2020}. These three types of $T_0$-spaces possess many important properties (see \cite{Lawson-Xu-2024-1, Lawson-Xu-2024-2, Li-Jin-Miao-Chen-2022, Xu-2026, Xu-Zhao-2020, Xu-Zhao-2021}), three of the key ones being: (1) the Scott space $\Sigma~\!\!P$ of a poset $P$ is a strong $d$-space iff it is strongly well-filtered (see \cite[Theorem 4.10]{Xu-2026}); (2) if $P$ is a dcpo such that $\Sigma P$ is strongly well-filtered and $\Sigma (P\times P)=\Sigma P\times\Sigma P$, then $\Sigma P$ is sober (see \cite[Theorem 5.11]{Lawson-Xu-2024-1}); and (3) a $T_0$-space $X$ is an $R$-space iff all closed subsets of $X$ are compact in the lower topology of $X$ equipped with the specialization order (see \cite[Theorem 35]{Lawson-Xu-2024-2}).

In order to address two open problems concerning strongly well-filtered spaces posed in \cite{Xu-2026}, Xu et al. \cite{Xu-Yang-Chen-2026} introduced a new class of $T_0$-spaces --- strong $R$-spaces, which are stronger than both $R$-spaces and strongly well-filtered spaces. Some relations among $T_2$-spaces, $T_1$-spaces, strong $R$-spaces, $R$-spaces, strongly well-filtered spaces, well-filtered spaces and sober spaces were investigated in \cite{Xu-Yang-Chen-2026}.

The main purpose of this paper is to investigate some basic properties of strong $R$-spaces. The paper is organized as follows.

Section 2 provides some fundamental definitions and notations about topology and lattice-ordered structures which will be used in the whole paper. Also a few basic properties of the Alexandroff topology on posets and reflective full subcategories of $\mathbf{Top}_0$ (the category of $T_0$-spaces and continuous mappings) are listed.

In Section 3, we briefly recall the concepts and some basic results of (strong) $d$-spaces, (strongly) well-filtered spaces and sober spaces. It is shown that every Noetherian poset equipped with the Alexandroff topology is strongly well-filtered.

In Section 4, we first recall the notions of $R$-spaces and strong $R$-spaces. Then we give two examples to show that a strongly well-filtered space may not be an $R$-space and a $T_1$-space (hence an $R$-space) may not be a strong $R$-space. It is shown that every complete lattice equipped with the Scott topology is a strong $R$-space. We prove that a $T_0$-space $(X, \tau)$ is a strong $R$-space iff every nonempty $\tau$-closed subset of $X$ is compact in $(X, \tau^{d})$, where $\tau^d$ is the de Groot dual of $\tau$; consequently, if $(X, \tau)$  is a strong $R$-space (especially, if $(X, \tau)$ is a coherent well-filtered space), then $\tau \subseteq \tau^{dd}$. Hence for any locally compact strong $R$-space $(X, \tau)$, we have $\tau=\tau^{dd}$, and the condition that $(X,\tau)$ is a strong $R$-space can not be weakened to
that of $(X,\tau)$ is an $R$-space or the strong well-filteredness of $(X,\tau)$.

Section 5 is mainly devoted to investigate some basic properties of strong $R$-spaces. It is shown that the property of being a strong $R$-space is closed-hereditary and saturated-hereditary, and every retract of a strong $R$-space is a strong $R$-space. We give two Scott spaces which are strong $R$-spaces but their product space is not a strong $R$-space. Hence the category $\mathbf{S}$-$\mathbf{Top}_r$ of strong $R$-spaces and continuous mappings is not reflective in $\mathbf{Top}_0$.

In Section 6, we investigate conditions under which the Smyth power space and Scott power space of a $T_0$-space is a strong $R$-space.  Several such conditions are given. In particular, we prove that for a coherent well-filtered space $X$, its Scott power space $P_{\sigma}(X)$  is a coherent well-filtered space, and consequently, $P_{\sigma}(X)$ is a strong $R$-space. It is shown that if the Smyth power space $P_S(X)$ of a $T_0$-space $X$ is an R-space, then $X$ is a strong $R$-space. So the Smyth power construction does not preserve the property of being an $R$-space in general. A second-countable Noetherian $T_0$-space $X$ is given for which the Scott power space $P_{\sigma}(X)$ is a strong $R$-space but $X$ is not a strong $R$-space.

In the final section, we present conclusions and some further work. Two questions concerning whether the Smyth power construction and Scott power construction preserve the property of being a strong $R$-space are posed.

\section{Preliminaries}
In this section, we briefly recall some basic concepts and results about topology and domain theory that will be used in the paper. For further details, we refer the reader to \cite{Abramsky-Jung-1994, GHKLMS-2003, Goubault-2013}.

For a poset $P$ and $A \subseteq P$, define ${\ua} A=\{x\in P: a\leq x \mbox{ for some }a\in A \}$ and ${\da} A=\{x \in P: x \leq a \mbox{ for some }a\in A\}$. For $x\in X$, let ${\ua} x={\ua}\{x\}$ and ${\da} x={\da}\{x\}$. A subset $A$ is called a \emph{lower set} (resp., an \emph{upper set}) if $A={\da} A$ (resp., $A={\ua} A$). Let $\mathbf{down}~\!P=\{{\downarrow} A :  A \mbox{ is a nonempty set of } P\}$ and  $\mathbf{Fin}~\!Q=\{{\ua} F : F\mbox{ is a nonempty finite set of } P\}$. For a nonempty subset $B$ of $P$, define $\mathrm{max}(B)=\{b\in B : b\mbox{~ is a maximal element of~} B\}$ and $\mathrm{min}(B)=\{b\in B : b \mbox{~ is a minimal element of~} B\}$. The set of all natural numbers is denoted by $\mathbb N$. Let $\mathbb N^+=\mathbb N\setminus \{0\}$. For a set $X$, let $|X|$ be the cardinality of $X$ and let $\omega=|\mathbb{N}|$. We use $X^{(<\omega)}$ (resp., $2^X$) to denote the set of all nonempty finite subsets (resp., all subsets) of $X$.

 For a poset $Q$ and $A\subseteq Q$, if the subset of upper bounds of $A$ has a unique smallest element, we call this element the \emph{least upper bound} of $A$ or the \emph{supremum} of $A$ and write it as $\vee A$ or sup\! $A$. Dually, the \emph{greatest lower bound} of $A$ is written as $\wedge A$ or inf\! $A$. The poset $Q$ is called a \emph{sup semilattice} if for any two elements $a, b\in Q$, $a\vee b$ exists in $Q$. A nonempty subset $D$ of a poset $Q$ is called \emph{directed} if every finite subset of $D$ has an upper bound in $D$. Dually, a nonempty subset $F$ of $Q$ is said to be \emph{filtered} if every finite subset of $F$ has a lower bound in $F$. The set of all directed subsets of $Q$ is denoted by $\mathcal{D}(Q)$. The poset $Q$ is called a \emph{directed complete poset}, or \emph{dcpo} for short, provided that the least upper bound $\vee D$ of $D$ exists in $Q$ for any $D\in\mathcal D(Q)$. The poset $Q$ is said to be \emph{sup-complete} if every nonempty subset $A$ has a supremum $\vee A$ in $Q$. Clearly, a poset $Q$ is sup-complete iff $Q$ is both a sup semilattice and a dcpo.

For a topological space $X$, we use $\mathcal O(X)$ (resp., $\Gamma(X)$) to denote the set of all open subsets (resp., all closed subsets) of $X$. The closure of a subset $A$ in $X$ will be denoted by $\cl_X A$ (or simply by $\cl~\! A$ if there is no ambiguity) or $\overline{A}$, and the interior of $A$ will be denoted by $\ii_X A$ or simply by $\ii~\! A$. For a subset $M$ of $X$, the \emph{induced topology} or \emph{subspace topology} on $M$ is denoted by $\mathcal O(X)|_M$, that is, $\mathcal O(X)|_M=\{U\cap M : U\in \mathcal O(X)\}$. A topological space $Y$ is said to be a \emph{Noetherian space} if every open subset of $Y$ is compact (see \cite[Definition 9.7.1]{Goubault-2013}).

\begin{definition}\label{def-continuous-domain} (\cite{GHKLMS-2003}) Let $P$ be a dcpo and $x, y\in P$. We say that $x$ is \emph{way below} $y$, in symbols $x\ll y$, if for every $D\in \mathcal D(P)$, $y\leq\mathrm{sup}~\! D$ implies $x\leq d$ for some $d\in D$. Let ${\Downarrow} x = \{u\in P : u\ll x\}$ and  $K(P)=\{k\in P : k\ll k\}$. Points of $K(P)$ are called \emph{compact elements}.
\begin{enumerate}[\rm (1)]
\item $P$ is called a \emph{continuous domain}, if for every $x\in P$, ${\Downarrow} x$ is directed
and $x=\mathrm{sup}~\!{\Downarrow} x$.
\item $P$ is called an \emph{algebraic domain}, if for every $x\in P$, $K(P)\cap {\downarrow} x$ is directed and $x=\mathrm{sup}~\! (K(P)\cap {\downarrow} x)$.
\item A topological space $X$ is called \emph{core-compact} if $\mathcal O(X)$ is a continuous lattice.
\end{enumerate}
\end{definition}

The following result is well-known (see, e.g.,  \cite[Proposition I-4.3]{GHKLMS-2003}).

\begin{proposition}\label{prop-algebraic-is-continuous} Every algebraic domain is a continuous domain.
\end{proposition}

\begin{remark}\label{rem-Hofmann-Lawson} It is easy to verify that every locally compact space is core-compact (see, e.g., \cite[Examples I-1.7]{GHKLMS-2003}). In \cite{Hofmann-Lawson-1978} (see also \cite[Exercise V-5.25]{GHKLMS-2003}) Hofmann and Lawson gave a second-countable core-compact $T_0$-space $X$ in which every compact subset of $X$ has empty interior, and hence it is not locally compact.
\end{remark}

We use $\mathbf{Top}_0$ to denote the category of $T_0$-spaces and continuous mappings. For a $T_0$-space $X$, let $\leq_X$ denote the \emph{specialization order} of $X$: $x\leq_X y$ if{}f $x\in \overline{\{y\}}$. In what follows,, when a $T_0$-space $X$ is considered as a poset, the order always refers to the specialization order if no other explanation is given.

As in \cite{GHKLMS-2003}, the \emph{upper topology} on a poset $P$, generated by the family $\{P\setminus {\da}x : x\in P\}$ (as a subbase), is denoted by $\upsilon(P)$. A subset $U$ of a poset $P$ is \emph{Scott open} if (i) $U={\uparrow}U$, and (ii) for any directed subset $D$ for which $\vee D$ exists, $\vee D\in U$ implies $D\cap U\neq\emptyset$. All Scott open subsets of $P$ form a topology, and we call this topology the \emph{Scott topology} on $P$ and denote it by $\sigma(P)$. The space $\Sigma P=(P, \sigma(P))$ is called the \emph{Scott space} of $P$. The upper sets of $P$ form the (\emph{upper}) \emph{Alexandroff topology} $\alpha (P)$.

A subset $A$ of a $T_0$-space $X$ is called \emph{saturated} if $A$ equals the intersection of all open sets containing it (equivalently, $A$ is an upper set in the specialization order). We use $Q(X)$ to denote the set of all nonempty compact saturated subsets of $X$ and endow it with the Smyth order $\sqsubseteq:~K_1\sqsubseteq K_2$ iff $K_{2}\subseteq K_{1}$. The space $X$ is said to be \emph{coherent} if the intersection of any two compact saturated sets is compact. For a $T_0$-topology $\tau$ on a set $Y$, we define the (\emph{de Groot}) \emph{dual} $\tau^d$ of the original topology $\tau$ by taking as a basis for the closed sets all compact saturated sets of $(Y, \tau)$ and  let $\tau^{dd}=(\tau^d)^d$ (see \cite{Groot-1966, Korpperman-1995, Kovar-2003}).

The following result is well-known and can be easily verified (cf. \cite[Section 3.2]{Heckmann-Keimel-2013}).

\begin{lemma}\label{lem-compact-saturated-in-Alexanderoff-topologyin} For a poset $P$, $Q((P, \alpha (P)))=\mathbf{Fin}~\!P$.
\end{lemma}

For a topological space $X$, $\mathcal{G}\subseteq 2^{X}$ and $A\subseteq X$, let $\Diamond_{\mathcal{G}}A=\{G\in \mathcal{G}: G\cap A\neq\emptyset\}$ and $\Box_{\mathcal{G}}A=\{G\in \mathcal{G}: G\subseteq A\}$. The sets $\Diamond_{\mathcal{G}}A$ and $\Box_{\mathcal{G}}A$ will be simply written as $\Diamond A$ and $\Box A$ respectively if there is no confusion. The \emph{upper Vietoris topology} on $\mathcal{G}$ is the topology that has $\{\Box_{\mathcal{G}} U : U\in \mathcal O(X)\}$ as a base, and the resulting space is denoted by $P_S(\mathcal{G})$. The space $P_S(\mathord{Q}(X))$, denoted shortly by $P_S(X)$, is called the \emph{Smyth power space} or \emph{upper power space} of $X$ (cf. \cite{Heckmann-1992, Heckmann-Keimel-2013, Schalk-1993}). The \emph{lower Vietoris topology} on $\mathcal{G}$ is the topology that has $\{\Diamond U: U\in \mathcal{O}(X)\}$ as a subbase, and the resulting space is denoted by $P_{H}(\mathcal{G})$. If $\mathcal{G}\subseteq \ir (X)$, then $\{\Diamond_{\mathcal{G}}U: U\in \mathcal{O}(X)\}$ is a topology on~$\mathcal{G}$. The space $P_H(\ir_{c}(X))$ is called the \emph{Hoare power space} or \emph{lower  power  space} of $X$ and is denoted by $P_H(X)$ for short (cf. \cite{Schalk-1993}), and $X^s=P_H(\ir_c(X))$ with the canonical mapping $\eta_{X}: X\longrightarrow X^s$ is called the \emph{canonical sobrification} of $X$ (see \cite{GHKLMS-2003, Goubault-2013}).

For two topological spaces $X$ and $Y$, $Y$ is said to be a \emph{retract} of $X$ if there are two continuous mappings $f : X\rightarrow Y$ and $g : Y\rightarrow X$ such that $f\circ g=id_{Y}$.

We have the following folklore result which can be easily verified.

\begin{lemma}\label{prod-retract}
	Let	$X=\prod_{i\in I}X_i$ be the product space of $T_0$-spaces  $X_i$ $(i\in I)$. Then for each $i\in I$, $X_i$ is a retract of $X$.
\end{lemma}

A topological property $S$ is said to be \emph{retractive} if for any space $X$ that has the property $S$, every retract of $X$ also has the property $S$.

 \begin{definition}(\cite[Definition 1.19]{Xu-2026})
Let $\mathbf{K}$ be a full subcategory $\mathbf{Top}_0$ and $X$ be a $T_{0}$ space. A $\mathbf{K}$-\emph{reflection} of $X$ is a pair $\langle \widetilde{X},\eta_X \rangle$ consisting of a $\mathbf{K}$-space (i.e., an object of $\mathbf{K}$) $\widetilde{X}$ and a continuous mapping $\eta_X :X\rightarrow \widetilde{X}$ satisfying that for any continuous mapping $f:X\rightarrow Y$ to a $\mathbf{K}$-space, there exists a unique continuous mapping $f^{*}:\widetilde{X}\rightarrow Y$ such that $f^{*}\circ \eta_X =f$, that is, the following diagram commutes.

\begin{equation*}
\centerline{
\xymatrix{ X \ar[dr]_{f} \ar[r]^-{\eta_X}&  \widetilde{X}\ar@{.>}[d]^{f^{*}} & \\
  & Y  & &
   }}
\end{equation*}
The category $\mathbf{K}$ is called \emph{reflective} in $\mathbf{Top}_0$ if every $T_0$-space has a $\mathbf{K}$-reflection.
\end{definition}



A full subcategory $\mathbf{K}$ of $\mathbf{Top}_0$ is called \emph{productive} (resp., \emph{finite} \emph{productive}) if the product space $\prod_{i\in I}X_i$ in $\mathbf{Top}_0$ is a $\mathbf{K}$-space whenever all spaces $X_i$ (resp., finitely many spaces $X_i$) are $\mathbf{K}$-spaces (i.e., objects of $\mathbf{K}$). A topological property $S$ is said to be \emph{productive} (resp., \emph{finite} \emph{productive}) if the product space $\prod_{i\in I}X_i$ in $\mathbf{Top}_0$ has the property $S$ whenever all $T_0$-spaces $X_i$ (resp., finitely many $T_0$-spaces $X_i$) have the property $S$. Let $\mathbf{Top}_S$ be the category of all $T_0$-spaces having the property $S$ and continuous mappings. Then the property $S$ is productive (resp., finite productive) iff  $\mathbf{Top}_S$ is productive (resp., finite productive).

By a standard categorical argument, we get the following result (see \cite[Remark 1.1]{Nel-Wilson-1972} or \cite[page 92, Exercise 7]{MacLane-1997} or \cite[Remark 4.8]{Keimel-Lawson-2009}).

\begin{proposition}\label{prop-category-reflective-productive} For a full subcategory $\mathbf{K}$ of $\mathbf{Top}_0$, if $\mathbf{K}$ is reflective in $\mathbf{Top}_0$, then $\mathbf{K}$ is productive.
\end{proposition}

\section{Sober spaces, $d$-spaces and well-filtered spaces}

  A $T_0$-space $X$ is called a \emph{$d$-space} (or \emph{monotone convergence space}) if $X$ (with the specialization order) is a dcpo and $\mathcal O(X) \subseteq \sigma(X)$ (cf. \cite{GHKLMS-2003, Wyler-1981}). The category of $d$-spaces and continuous mappings is denoted by $\mathbf{Top}_d$.

A nonempty subset $A$ of a $T_0$-space $X$ is said to be \emph{irreducible} if for any $F_1$, $F_2\in\Gamma(X)$, $A\subseteq F_1\cup F_2$ implies $A\subseteq F_1$ or $A\subseteq F_2$. We denote by $\ir(X)$ (resp., $\ir_c(X)$) the set of all irreducible subsets (resp., all irreducible closed subsets) of $X$. Clearly, every directed subset of $X$ (with the specialization order) is irreducible and the nonempty irreducible subsets of a poset equipped with the Alexandroff topology are exactly the directed sets of $P$ (cf. \cite[Fact 2.6]{Heckmann-Keimel-2013} or \cite[Lemma 1.2]{Hoffmann-1979}). The space $X$ is called \emph{sober}, if for any $A\in\ir_c(X)$, there is a unique point $x\in X$ such that $A=\overline{\{x\}}$. We denote by $\mathbf{Sob}$ the category of sober spaces and continuous mappings.

The following conclusion is well-known (see, e.g., \cite[Corollary II-1.12]{GHKLMS-2003}).

\begin{proposition}\label{prop-continuous-domain-Scott-is-sober} For a continuous domain $P$, $\Sigma ~\!\! P$ is sober.
\end{proposition}

For the sobriety of upper topology on a poset, we have the following result.

\begin{proposition}\label{prop-upper-topology-WF} (\cite[Proposition 2.9]{Xu-Shen-Xi-Zhao-2020-2}) For a poset $P$, the space $(P, \upsilon (P))$ is sober iff $\vee A$ exists in $P$ for any $A\in\ir ((P, \upsilon (P))$. Therefore, for any sup-complete poset (especially, any  complete lattice) $L$, $(L, \upsilon (L))$ is sober.
\end{proposition}

For a $T_0$-space $X$, it is easy to verify that $P_H(X)=(\Gamma(X) \setminus\{\emptyset\}, \upsilon(\Gamma(X)\setminus\{\emptyset\}))$. So by Proposition \ref{prop-upper-topology-WF} we get the following.

\begin{proposition}\label{prop-Hoare-is-sober}  For a $T_0$-space $X$, its Hoare power space $P_H(X)$ is sober.
\end{proposition}

For the sobriety of the Smyth power spaces, we have the following well-known result.

\begin{theorem}\label{theor-Schalk-Heckman-Keimel-theorem} (Heckmann-Keimel-Schalk Theorem) (\cite[Theorem 3.13]{Heckmann-Keimel-2013} and \cite[Lemma 7.20]{Schalk-1993}) For a $T_0$-space $X$, the following conditions are equivalent:
\begin{enumerate}[\rm (1)]
\item $X$ is sober.
 \item  For any $\mathcal A\in \ir(P_S(X))$ and $U\in \mathcal O(X)$, $\bigcap\mathcal A\subseteq U$ implies $K \subseteq U$ for some $K\in \mathcal A$.
 \item $P_S(X)$ is sober.
\end{enumerate}
\end{theorem}

 A $T_0$-space $X$ is called \emph{well-filtered} if for any filtered family $\mathcal{K}\subseteq \mathord{Q}(X)$ and $U\in \mathcal O(X)$, $\bigcap\mathcal{K}{\subseteq} U$ implies $K{\subseteq} U$ for some $K{\in}\mathcal{K}$. We use $\mathbf{Top}_w$ to denote the category of well-filtered spaces and continuous mappings.

For the well-filteredness and coherence of Scott spaces, we have the following two useful results.

\begin{proposition}\label{prop-Scott-topology-on-complete-lattice-WF} (\cite[Corollary 3.2]{Xi-Lawson-2017}) For a complete lattice $L$, its Scott space $(L, \sigma (L))$ is well-filtered.
\end{proposition}

\begin{proposition}\label{prop-Scott-WF-coherent} (\cite[Lemma 2.1 and Corollary 3.2]{Jia-Jung-2016}) Let $P$ be a dcpo for which $\Sigma~\!\!P$ is well-filtered. Then $\Sigma~\!\!P$ is coherent if and only if ${\ua}x\cap {\ua}y$ is compact in $\Sigma~\!\!P$ for all $x, y\in P$. Therefore, for any complete lattice $L$, the Scott space $\Sigma~\!\!L$ is coherent.
\end{proposition}

\begin{proposition}\label{prop-CI-WF-sober} (\cite[Theorem 4.2]{Xu-Shen-Xi-Zhao-2020-1})
	Let $X$ be a first-countable well-filtered space. Then $X$ is sober.
\end{proposition}

For the well-filteredness, a similar result to Theorem \ref{theor-Schalk-Heckman-Keimel-theorem} was obtained in \cite{Xu-Shen-Xi-Zhao-2020-2, Xu-Xi-Zhao-2021}.

\begin{theorem}\label{Smythwf} (\cite[Theorem 5.3]{Xu-Shen-Xi-Zhao-2020-2} or \cite[Theorem 4]{Xu-Xi-Zhao-2021})
	For a $T_0$-space $X$, the following conditions are equivalent:
\begin{enumerate}[\rm (1)]
		\item $X$ is well-filtered.
        \item $P_S(X)$ is a $d$-space.
        \item $P_S(X)$ is well-filtered.
\end{enumerate}
\end{theorem}

The following two results are well-known (see, e.g., \cite{GHKLMS-2003, Kou-2001, Lawson-Wu-Xi-2020, Xu-Shen-Xi-Zhao-2020-2}, \cite[Theorem 3.4]{Wyler-1981}, \cite[IV 4.2.1]{Artin-Grothendieck-1972} and \cite[Corollary 6.9]{Xu-Shen-Xi-Zhao-2020-2}).

\begin{theorem}\label{theor-LC-sober=LC-wf=CC-sober}  For a $T_0$-space $X$, the following conditions are equivalent:
\begin{enumerate}[\rm (1)]
	\item $X$ is locally compact and sober.
	\item $X$ is locally compact and well-filtered.
	\item $X$ is core-compact and sober.
   \item $X$ is core-compact and well-filtered.
\end{enumerate}
\end{theorem}

\begin{theorem}\label{theor-sober-WF-d-space-reflective} $\mathbf{Sob}$, $\mathbf{Top}_d$ and $\mathbf{Top}_w$ are all reflective in $\mathbf{Top}_0$.
\end{theorem}

A poset $P$ is said to be \emph{Noetherian} if it satisfies the \emph{ascending chain condition} ($\mathrm{ACC}$ for short): every ascending chain has a greatest member. It is easy to verify that $P$ is Noetherian if{}f every directed set of $P$ has a largest element. So every Noetherian poset is a dcpo.

\begin{proposition}\label{prop-Alexandroff-topology-sober} (\cite[Proposition 11]{Lawson-Xu-2024-2}).
	For a poset $P$, the following conditions are equivalent:
	\begin{enumerate}[\rm (1)]
		\item $P$ is Noetherian.
        \item $P$ is a dcpo and $\alpha(P)=\sigma(P)$.
        \item $P$ is a dcpo and $x\in K(P)$ for any $x\in P$.
\item $P$ is a dcpo and ${\uparrow}x\in \sigma(P)$ for any $x\in P$.
       \item $(P,\alpha(P))$ is sober.
		\item $(P,\alpha(P))$ is well-filtered.
		\item $(P,\alpha(P))$ is a $d$-space.
	\end{enumerate}
\end{proposition}

As strengthened versions of $d$-spaces and well-filtered spaces, the following two notions were introduced in \cite{Xu-2026, Xu-Zhao-2020}.

\begin{definition}\label{strong $d$-space} (\cite[Definition 3.18]{Xu-Zhao-2020}) A $T_0$-space $X$ is called a \emph{strong} $d$-\emph{space} if for any $D\in \mathcal D(X)$, $x\in X$ and $U\in \mathcal O(X)$, $\bigcap_{d\in D}\ua d\cap \ua x\subseteq U$ implies $\ua d\cap \ua x\subseteq U$ for some $d\in D$.
\end{definition}

\begin{definition}(\cite[Definition 4.1]{Xu-2026}) \label{def-strong-WF} A $T_0$-space $X$ is called \emph{strongly well}-\emph{filtered} if for any filtered family $\{K_d : d\in D\}\subseteq Q(X)$, $K\in Q(X)$ and $U\in \mathcal O(X)$, $\bigcap_{d\in D}K_d\cap K\subseteq U$ implies $K_d\cap K\subseteq U$ for some $d\in D$.
\end{definition}

\begin{proposition}\label{prop-strongly-WF-is-WF} (\cite[Proposition 4.3]{Xu-2026}) \emph{(1)} Every strongly well-filtered space is well-filtered.

\noindent \emph{(2)} Every strongly well-filtered space is a strong $d$-space.

\noindent \emph{(3)} Every coherent well-filtered space is strongly well-filtered.

\noindent \emph{(4)} Every $T_2$-space is coherent and sober and hence strongly well-filtered.
\end{proposition}

\begin{lemma}\label{lem-Scott-sd-space-charac} (\cite[Lemma 3.25]{Xu-Zhao-2020}) For a poset $P$, the following two conditions are equivalent:
 \begin{enumerate}[\rm (1)]
 \item $\Sigma~\!\!P$ is a strong $d$-space.
 \item $P$ is a dcpo, and for any $x\in P$ and $A\in \Gamma(\Sigma~\!\!P)$, $\da (\ua x\cap A)\in \Gamma(\Sigma~\!\!P)$.
 \end{enumerate}
 \end{lemma}

 As we all know, the property of being a $d$-space is strictly weaker than the well-filteredness (see, e.g., \cite[Exercise 8.3.9]{Goubault-2013} or \cite[Example 3.1]{Lu-Li-2017}). In contrast with it, Xu \cite{Xu-2026} gave the following important and unexpected result.

\begin{theorem}\label{theor-Scott-strongly-WF-strong-d-space} (\cite[Theorem 4.10]{Xu-2026}) For a poset $P$, the following two conditions are equivalent:
 \begin{enumerate}[\rm (1)]
 \item $\Sigma~\!\!P$ is a strong $d$-space.
 \item $\Sigma~\!\!P$ is strongly well-filtered.
 \end{enumerate}
 \end{theorem}

\begin{proposition}\label{prop-Noetherian-Scott-space-strongly-WF} For a Noetherian poset $P$, $\Sigma~\!\!P$ is strongly well-filtered.
 \end{proposition}
\begin{proof} By Proposition \ref{prop-Alexandroff-topology-sober}, $P$ is a dcpo and $\alpha(P)=\sigma(P)$, whence $\Gamma (\Sigma~\!\!P)=\mathbf{down}~\!P$ and $Q(\Sigma~\!\!P)=\mathbf{Fin}~\!P$ by Lemma \ref{lem-compact-saturated-in-Alexanderoff-topologyin}. Let $x\in P$ and $A\in \Gamma(\Sigma~\!\!P)$. Then $\da (\ua x\cap A)\in \mathbf{down}~\!P=\Gamma (\Sigma~\!\!P)$. Hence by Lemma \ref{lem-Scott-sd-space-charac} $\Sigma~\!\!P$ is a strong $d$-space, and consequently, it is strongly well-filtered by Theorem \ref{theor-Scott-strongly-WF-strong-d-space}.
\end{proof}

\section{Strong $R$-spaces and de Groot dual topology}

\begin{definition}\label{def-$R$-space} A $T_0$-space $X$ is said to have \emph{property R} if for any nonempty family $\{\mathord{{\uparrow}}F_{i}:i\in I\}\subseteq \mathbf{Fin}~\!\! X$ and $U\in \mathcal O(X)$, $\bigcap_{i\in I}\mathord{{\uparrow}}F_{i}\subseteq U$ implies  $\bigcap_{i\in J}\mathord{{\uparrow}}F_{i}\subseteq U$ for some $J\in I^{(<\omega)}$.
\end{definition}

The property $R$, which is strictly stronger than the property of being a strong $d$-space, was first introduced in \cite[Definition 10.2.11]{Xu-2016-2} (see also \cite[Definition 2.2]{Wen-Xu-2018}). Lawson and Xu \cite{Lawson-Xu-2024-2} called a $T_0$-space $X$ that satisfies property $R$ an \emph{R}-\emph{space}. The $R$-spaces were systematically studied in \cite{Lawson-Xu-2024-2, Xu-2016-2}.

As a strengthened version of $R$-spaces, Xu et al. \cite{Xu-Yang-Chen-2026} introduced the following notion.

\begin{definition}\label{def-strong-$R$-space}(\cite[Definition 4.2]{Xu-Yang-Chen-2026}) A $T_0$-space $X$ is called a \emph{strong} \emph{R}-\emph{space} if for any nonempty family $\{K_{i} : i\in I\}\subseteq Q(X)$ and any $U\in \mathcal O(X)$, $\bigcap_{i\in I}K_i\subseteq U$ implies $\bigcap_{i\in J}K_i\subseteq U$ for some $J\in I^{(<\omega)}$. The category of strong $R$-spaces and continuous mappings is denoted by $\mathbf{S}$-$\mathbf{Top}_r$.
\end{definition}

The essential distinction between strong $R$-spaces and $R$-spaces is that the definition of an $R$-space uses the system of upper sets generated by nonempty finite sets, while that of a strong $R$-space uses the system of nonempty compact saturated subsets. Clearly, every strong $R$-space is an $R$-space and the Sierpi\'{n}ski space $\Sigma 2$ is a strong $R$-space but not $T_1$.

\begin{proposition}\label{prop-strong-R-is-strong-WF} (\cite[Proposition 4.3]{Xu-Yang-Chen-2026}) \noindent \emph{(1)} Every strong $R$-space is strongly well-filtered.

\noindent \emph{(2)} Every coherent well-filtered space is a strong $R$-space.

\noindent \emph{(3)} Every $T_2$-space is a coherent sober space and hence a strong $R$-space.

\noindent \emph{(4)} Every $T_1$-space is an $R$-space.
\end{proposition}

The following example shows that a strongly well-filtered space may not be an $R$-space. It also shows that for a Noetherian dcpo $P$, $(P, \alpha(P))$ may not be an $R$-space (comparing it with Proposition \ref{prop-Noetherian-Scott-space-strongly-WF}).

\begin{example}\label{exam-strongly-WF-is-not-$R$-space} Let $P=\{a_n : n\in \mathbb{N}^+\}\cup\{\omega_n : n\in\mathbb{N}^+\}\cup\{\top\}$  (see Figure 1) with the order generated by
\begin{enumerate}[\rm (1)]
            \item $\omega_n<\top$ for all $n\in\mathbb{N}^+$; and
            \item $a_n<\omega_m$ for all $n, m\in \mathbb{N}$ with $n\leq m$.
\end{enumerate}

\begin{figure}[ht]
	\centering
	\includegraphics[height=4.5cm,width=10.5cm]{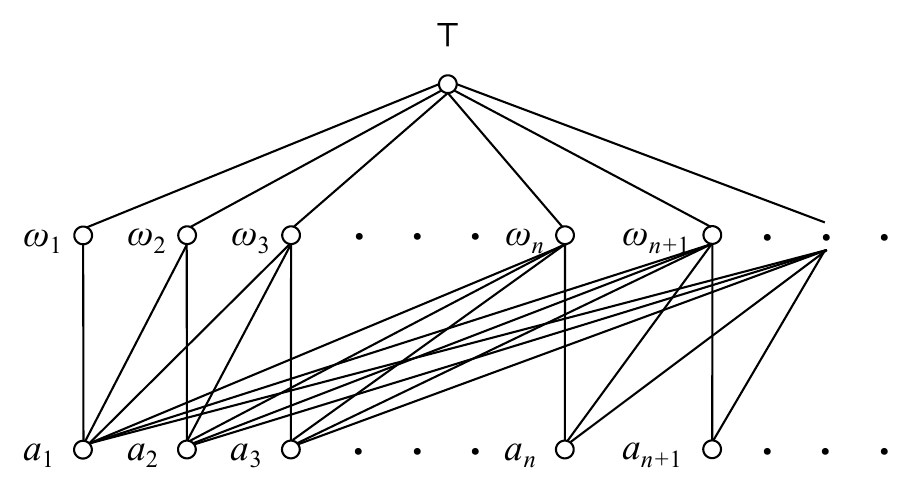}
	\caption{The Noetherian dcpo $P$ in Example \ref{exam-strongly-WF-is-not-$R$-space}}
\end{figure}

Considering the Scott topology $\sigma(P)$, we have the following conclusions:
\begin{enumerate}[\rm (a)]
\item $P$ is Noetherian and hence it is a dcpo and $\sigma(P)=\alpha(P)$ by Proposition \ref{prop-Alexandroff-topology-sober}.
\item $\Gamma (\Sigma~\!\!P)=\mathbf{down}~\!P$ and $Q(\Sigma~\!\!P)=\mathbf{Fin}~\!P$ by Lemma \ref{lem-compact-saturated-in-Alexanderoff-topologyin}.
    \item $\Sigma~\!\!P$ is sober by Proposition \ref{prop-Alexandroff-topology-sober}.
\item $\Sigma~\!\!P$ is strongly well-filtered by $\sigma(P)=\alpha(P)$ and Proposition \ref{prop-Noetherian-Scott-space-strongly-WF}.
\item $\Sigma~\!\!P$ is not an $R$-space and hence it is not a strong $R$-space.

Clearly, $\{{\uparrow}a_n : n\in\mathbb{N}^+\}\subseteq Q(\Sigma~\!\!P)$ and $\bigcap\limits_{n\in\mathbb{N}^+}{\uparrow}a_n=\{\top\}$, but for any $J\in (\mathbb{N}^+)^{(<\omega)}$, $\bigcap_{n\in J}{\uparrow}a_n\supseteq \{\top\}\cup\{\omega_m : m\geq max(J)\}\not\subseteq \{\top\}$. Therefore, $\Sigma~\!\!P$ is not an $R$-space, and hence it is not a strong $R$-space.
\end{enumerate}
\end{example}

A $T_1$-space (hence an $R$-space) may not be a strong $R$-space, as shown in the following example.

\begin{example}\label{exam-X-cof-T1-space-not-strong-$R$-space}
	Let $X$ be a countably infinite set and $X_{cof}$ the space equipped with the \emph{co-finite topology} (the empty set and the complements of finite subsets of $X$ are open). Then
\begin{enumerate}[\rm (a)]
    \item $X_{cof}$ is a $T_1$-space and hence an $R$-space.
    \item $Q(X_{cof})=2^X\setminus \{\emptyset\}$.
    \item $X_{cof}$ is locally compact and first-countable.
\item $X_{cof}$ is not well-filtered and hence it is not a strong $R$-space.

Let $\mathcal K=\{X\setminus F: F\in X^{(<\omega)}\}$. Then $\mathcal K$ is a filtered family of saturated compact subsets of $X_{cof}$ and $\bigcap \mathcal K=\emptyset$, but $X\setminus F\neq\emptyset$ for every $ F\in X^{(<\omega)}$. Therefore, $X_{cof}$ is not well-filtered, and hence it is not a strong $R$-space.
\end{enumerate}
\end{example}

Figure 2 taken from \cite{Xu-Yang-Chen-2026} shows certain relations of some spaces lying between $d$-spaces and $T_2$-spaces.

\begin{figure}[ht]
	\centering
	\includegraphics[height=4.5cm,width=12.0cm]{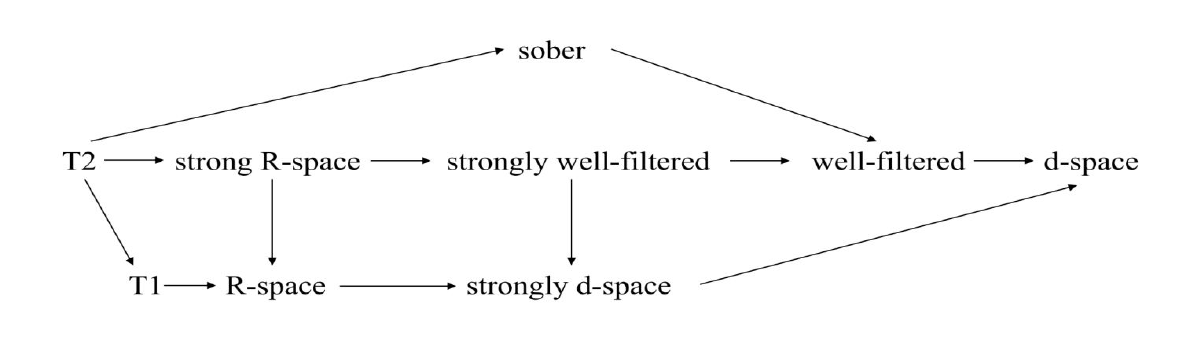}
	\caption{Relations of some spaces lying between $d$-spaces and $T_2$-spaces}
\end{figure}

By Proposition \ref{prop-Scott-topology-on-complete-lattice-WF}, Proposition \ref{prop-Scott-WF-coherent} and Proposition \ref{prop-strong-R-is-strong-WF}(2), we have the following result.

\begin{proposition}\label{prop-complete-lattice-Scott-space-strong-$R$-space} For a complete lattice $L$, its Scott space $\Sigma~\!\!L$ is a strong $R$-space.
\end{proposition}

\begin{proposition}\label{prop-compact-bigcap} (\cite[Theorem 4.5 and Proposition 4.7]{Xu-Yang-Chen-2026})
 For a sup-complete poset $P$, $(P, \upsilon(P))$ is a strong $R$-space and coherent. Hence $(P, \upsilon(P))$ is strongly well-filtered.
\end{proposition}

\begin{corollary}\label{cor-Hower-swf} (\cite[Theorem 4.6]{Xu-Yang-Chen-2026})
For a $T_0$-space $X$, its Hoare power space $P_H(X)$ is a strong $R$-space and coherent. Hence $P_H(X)$  is strongly well-filtered.
\end{corollary}

The following result demonstrates a close relation between the property of being a strong $R$-space and the de Groot dual topology.

\begin{theorem}\label{theor-strong-$R$-space-charac}
For a $T_0$-space $(X, \tau)$, the following two conditions are equivalent:
 \begin{enumerate}[\rm (1)]
 \item $(X, \tau)$ is a strong $R$-space.
 \item $\Gamma((X, \tau))\setminus \{\emptyset\}\subseteq Q((X, \tau^{d}))$.
 \end{enumerate}
\end{theorem}
\begin{proof} (1) $\Rightarrow$ (2): Suppose that $C$ is a nonempty closed set of $(X, \tau))$ and $\{K_{i}:i\in I\}\subseteq Q((X, \tau))$ satisfying $C\subseteq \bigcup_{i\in I}(X\setminus K_{i})$. Then $C$ is saturated in $(X, \tau^{d})$ (note the specialization order of $(X, \tau^{d})$ is the dual of specialization order of $(X, \tau)$) and $\bigcap_{i\in I}K_{i}\subseteq X\setminus C\in\tau$. As $(X, \tau)$ is a strong $R$-space, there is $J\in I^{(<\omega)}$ such that $\bigcap_{i\in J}K_{i}\subseteq X\setminus C$ or, equivalently, $C\subseteq \bigcup_{i\in I}(X\setminus K_{i})$. Therefore, $C$ is compact in $(X, \tau^d)$ by Alexander's Subbase Lemma (see, eg., \cite[Proposition I-3.22]{GHKLMS-2003}), and hence $C\in Q((X, \tau^{d}))$. Thus $\Gamma((X, \tau))\setminus \{\emptyset\}\subseteq Q((X, \tau^{d}))$.

(2) $\Rightarrow$ (1): Assume that $\{K_{i}:i\in I\}\subseteq Q((X, \tau))$ and $U\in\tau$ satisfying $\bigcap_{i\in I}K_{i}\subseteq U$. If $U=X$, then $\bigcap_{i\in I_0}K_{i}\subseteq U$ for any $I_0\in I^{(<\omega)}$. If $U\neq X$, then $C=X\setminus U$ is a nonempty closed set of $(X, \tau)$. By the hypothesis, $C\in \Gamma((X, \tau))\setminus \{\emptyset\}\subseteq Q((X, \tau^{d}))$ and $C\subseteq \bigcup_{i\in I}(X\setminus K_i)$. Therefore, there is $J\in I^{(<\omega)}$ such that $C\subseteq \bigcup_{i\in J} (X\setminus K_{i})$, or equivalently, $\bigcap_{i\in J}K_i\subseteq U$. Hence $X$ is a strong $R$-space.
\end{proof}

\begin{corollary}\label{cor-strong-$R$-space-tau-dd}
If $(X, \tau)$  is a strong $R$-space (especially, if $(X, \tau)$ is a coherent well-filtered space, or a $T_2$-space), then $\tau \subseteq \tau^{dd}$.
\end{corollary}

\begin{theorem}\label{theor-locally-compact-dual-topology}
Let $(X, \tau)$ be a locally compact $T_0$-space. Then $Q((X,\tau^{d}))\subseteq \Gamma((X, \tau))\setminus \{\emptyset\}$, and hence $\tau^{dd}\subseteq \tau$.
\end{theorem}
\begin{proof} Suppose $A\in Q((X,\tau^{d}))$. Then $A={\downarrow_{\tau}} A$. If $A=X$, then $A\in \Gamma((X, \tau))\setminus \{\emptyset\}$. Now we assume $X\setminus A\neq\emptyset$. For $x\in X\setminus A$, as $X$ is locally compact, $\mathcal K_x=\{K\in Q((X, \tau)) : x\in \ii_{\tau} K\}$ is filtered and $\bigcap\mathcal K_x=\ua_{\tau} x\subseteq X\setminus A$, whence $A\subseteq \bigcup_{K\in \mathcal K_x}(X\setminus K)$. Hence $\{X\setminus K : K\in \mathcal K_x\}$ is an open cover of $A$ in $(X,\tau^{d})$. So by $A\in Q((X,\tau^{d}))$, there exists $K\in \mathcal K_x$ such that $A\subseteq X\setminus K$ (note that $\{X\setminus K : K\in \mathcal K_x\}$ is a direct family), and consequently, $x\in \ii_{\tau} K\subseteq K\subseteq X\setminus A$. Therefore, $X\setminus A\in \mathcal{O}(X)\setminus \{X\}$, or equivalently, $A\in \Gamma((X, \tau))\setminus \{\emptyset\}$. Thus $Q((X,\tau^{d}))\subseteq \Gamma((X, \tau))\setminus \{\emptyset\}$, and hence $\tau^{dd}\subseteq \tau$.
\end{proof}

From Theorem \ref{theor-strong-$R$-space-charac} and Theorem \ref{theor-locally-compact-dual-topology} we deduce the following.

\begin{theorem}\label{theor-locally-compact-strong-$R$-space}
Let $(X,\tau)$ be a locally compact strong $R$-space. Then $\Gamma((X, \tau))\setminus \{\emptyset\}=Q((X,\tau^{d}))$, and hence $\tau=\tau^{dd}$.
\end{theorem}

By Theorem \ref{theor-LC-sober=LC-wf=CC-sober}, Proposition \ref{prop-strongly-WF-is-WF}(1), Proposition \ref{prop-strong-R-is-strong-WF}(1) and Theorem \ref{theor-locally-compact-strong-$R$-space}, we get the following.

\begin{corollary}\label{cor-co-compact-strong-$R$-space}
Let $(X,\tau)$ be a core-compact strong $R$-space. Then $\Gamma((X, \tau))\setminus \{\emptyset\}=Q((X,\tau^{d}))$, and hence $\tau=\tau^{dd}$.
\end{corollary}

The following corollary follows directly from  Proposition \ref{prop-strong-R-is-strong-WF}(2) and Theorem \ref{theor-locally-compact-strong-$R$-space}.

\begin{corollary}\label{cor-locally-compact-T2-space}
If $(X,\tau)$ is a locally compact coherent well-filtered space (especially, if $(X, \tau)$ is a locally compact $T_2$-space), then $\Gamma((X, \tau))\setminus \{\emptyset\}=Q((X,\tau^{d}))$, and hence $\tau=\tau^{dd}$.
\end{corollary}

In Corollary \ref{cor-strong-$R$-space-tau-dd} and Theorem \ref{theor-locally-compact-strong-$R$-space}, the condition that $(X,\tau)$ is a strong $R$-space can not be weakened to that of $(X,\tau)$ is an $R$-space or the strong well-filteredness of $(X,\tau)$, as shown in the following two examples.

\begin{example}\label{exam-X-cof-T1-locally-compact-tau-not-tau-dd} Let $\tau=\{\mathbb{N}\setminus F : F\in\mathbb{N}^{(<\omega)}\}\cup\{\emptyset, \{0\}\}$. Then
\begin{enumerate}[\rm (a)]
\item $\tau$ is a $T_1$-topology on $\mathbb{N}$ and hence $(\mathbb{N}, \tau)$ is an $R$-space.
\item $Q((\mathbb{N}, \tau))=2^{\mathbb{N}}\setminus \{\emptyset\}$.

Clearly, every finite subset of $\mathbb{N}$ is compact in $(\mathbb{N}, \tau)$. Assume that $A$ is an infinite subset of $\mathbb{N}$ and $\{U_i : i\in I\}\subseteq \tau$ is an open cover of $A$. Select $x\in A\setminus \{0\}$. Then $x\in U_j$ for some $j\in I$. Since $x\neq 0$, $U_j\neq \{0\}$, whence there is $G\in \mathbb{N}^{(<\omega)}$ with $U_j=\mathbb{N}\setminus G$. Therefore, $A\setminus U_j\subseteq G$ is finite, and hence there is $I_0\in I^{(<\omega)}$ such that $A\setminus U_j\subseteq \bigcup_{i\in I_0} U_i$. So $A\subseteq \bigcup_{i\in I_0\cup \{j\}} U_i$. Hence $A\in Q((\mathbb{N}, \tau))$. Thus $Q((\mathbb{N}, \tau))=2^{\mathbb{N}}\setminus \{\emptyset\}$.
\item $(\mathbb{N}, \tau)$ is locally compact and $\tau^d=2^{\mathbb{N}}$ by (b).
\item $\tau\not\subseteq\tau^{dd}$.

By $\tau^d=2^{\mathbb{N}}$, we have $\tau^{dd}=\{\mathbb{N}\setminus F : F\in \mathbb{N}^{(<\omega)}\}\cup \{\emptyset\}$. Then $\{0\}\in \tau\setminus \tau^{dd}$, and hence $\tau\not\subseteq \tau^{dd}$.
\end{enumerate}
\end{example}

\begin{example}\label{exam-strongly-WF-is-not-$R$-space-1}  Let $P$ be the countable Noetherian dcpo in Example \ref{exam-strongly-WF-is-not-$R$-space}. Consider the Scott topology $\sigma(P)$. Then we have the following conclusions (see Example \ref{exam-strongly-WF-is-not-$R$-space}):
\begin{enumerate}[\rm (a)]
\item $\sigma(P)=\alpha(P)$, $\Gamma (\Sigma~\!\!P)=\mathbf{down}~\!P$ and $Q(\Sigma~\!\!P)=\mathbf{Fin}~\!P$.
\item $\Sigma~\!\!P$ is strongly well-filtered but it is not a strong $R$-space.
\item $\Sigma~\!\!P$ is locally compact by (a).
\item $\{P\setminus {\uparrow}F : {\uparrow}F\in\mathbf{Fin}~\!P\}$ is a base of $\sigma (P)^d$ by $Q(\Sigma~\!\!P)=\mathbf{Fin}~\!P$.
\item $\sigma(P)\nsubseteq \sigma(P)^{dd}$.

Clearly, $P\setminus \{\top\}$ is a Scott closed set of $P$. Now we show that $P\setminus \{\top\}$ is not closed in $(P, \sigma(P)^{dd})$. Assume, on the contrary, that $P\setminus \{\top\}\in \Gamma ((P, \sigma(P)^{dd})$. Then there is a family $\{G_i : i\in I\}\subseteq Q((P, \sigma(P)^d))$ such that $P\setminus \{\top\}=\bigcap_{i\in I}G_i$ (note that $Q((P, \sigma(P)^d))$ is a basis for the closed sets of $(P, \sigma(P)^{dd}$)). So $P\setminus \{\top\}\subseteq G_i$ for all $i\in I$ and there is $j\in I$ with $G_j=P\setminus \{\top\}$, whence $P\setminus \{\top\}$ is compact in $(P, \sigma(P)^d)$. As $\bigcap\limits_{n\in\mathbb{N}^+}{\uparrow}a_n=\{\top\}$, $\{P\setminus {\uparrow}a_n : n\in\mathbb{N}^+\}$ is an open cover of $P\setminus \{\top\}$ in $(P, \tau^d)$ but has no finite subcover, which contradicts $P\setminus \{\top\}\in Q((P, \sigma(P)^d))$. Therefore, $P\setminus \{\top\}$ is not closed in $(P, \sigma(P)^{dd}$, and hence $\sigma(P)\nsubseteq \sigma(P)^{dd}$.
\end{enumerate}
\end{example}

\section{Some basic properties of strong $R$-spaces}

A topological property $S$ is said to be \emph{hereditary} (\emph{saturated}-\emph{hereditary}, \emph{closed}-\emph{hereditary}, \emph{open}-\emph{hereditary} etc.) if for any space $X$ that has the property $S$, every subspace (every saturated subspace,
closed subspace, open subspace etc.) of $X$ also has the property $S$.

First, we show that the property of being a strong $R$-space is closed-hereditary and saturated-hereditary.

\begin{proposition}\label{closed-$R$-space} Let $X$ be a strong $R$-space and $A$ a nonempty closed subset of $X$. Then $A$ as a closed subspace of $X$ is a strong $R$-space.
\end{proposition}
\begin{proof}
Suppose that $\{K_{i}:i\in I\}\subseteq Q(A)$ is a nonempty family and $U\in \mathcal{O}(X)\mid_A$ satisfying $\bigcap_{i\in I}K_{i}\subseteq U$. For each $i\in I$, we have ${\uparrow_X}K_{i}\in Q(X)$ and $K_i=A\cap \ua _{X}K_{i}$. As $U\in \mathcal{O}(A)$, there is $V\in \mathcal{O}(X)$ such that $U=V\cap A$. Then $\bigcap_{i\in I}{\ua_X}K_{i}\subseteq U\cup (X\setminus A)\subseteq V\cup (X\setminus A)\in \mathcal O(X)$. Since $X$ is a strong $R$-space, there exists $J\in I^{(<\omega)}$ such that  $\bigcap_{i\in J}{\ua_{X}}K_{i}\subseteq V\cup (X\setminus A)$. Hence $\bigcap_{i\in J}K_{i}=\bigcap_{i\in I_{0}}{\ua_{X}}K_{i}\cap A\subseteq (V\cup (X\setminus A))\cap A=U$. Thus $A$ as a closed subspace of $X$ is a strong $R$-space.
\end{proof}

\begin{proposition}\label{saturated-strongly-WF-is-also} Let $X$ be a strong $R$-space and $U$ a nonempty saturated subset of $X$. Then $U$ as a subspace of $X$ is a strong $R$-space.
\end{proposition}
\begin{proof}
Assume that $\{K_{i}:i\in I\}\subseteq Q(U)$ is a nonempty family and $V\in \mathcal{O}(X)\mid_U$ satisfying $\bigcap_{i\in I}K_{i}\subseteq V$. Then there exists $W\in \mathcal{O}(X)$ such that $V=W\cap U$, whence $\bigcap_{i\in I}K_{i}\subseteq W$. As $U=\ua_{X} U$, it is clear that $\{K_{i}:i\in I\}\subseteq Q(X)$. Since $X$ is a strong $R$-space, there exists $J\in I^{(<\omega)}$ such that $\bigcap_{i\in J}K_{i}\subseteq W$. Then $\bigcap_{i\in J}K_{i}\subseteq W\cap U=V$, and this completes the proof that $U$ as a subspace of $X$ is a strong $R$-space.
\end{proof}

A topological property $S$ is said to be \emph{retractive} if for any space $X$ that has the property $S$, every retract of $X$ also has the property $S$.

Now we show that the property of being a strong $R$-space is retractive.

\begin{proposition}\label{prop-retract-strongly-WF}
A retract of a strong $R$-space is a strong $R$-space.
\end{proposition}

\begin{proof}  It is well-known that a retract of a $T_0$-space is $T_0$ (cf. \cite[Lemma 4.10.32]{Goubault-2013}).
Let $X$ be a strong $R$-space and $Y$ a retract of $X$. Then $Y$ is a $T_0$-space and there are continuous mappings $f : X\rightarrow Y$ and $g: Y\rightarrow X$ with $f\circ g=id_Y$. Suppose that $\{G_i:i\in I\}\subseteq Q(Y)$ and $V\in \mathcal O(Y)$ satisfying $\bigcap_{i\in I}G_i\subseteq V$. Then $\{{\ua}~\!\!g(G_i) : i\in I\}\subseteq Q(X)$ and $f(\bigcap_{i\in I}{\ua}~\!\!g(G_i))\subseteq \bigcap_{i\in I}f({\ua}~\!\! g(G_i))\subseteq \bigcap_{i\in I}{\ua} f(g(G_i))=\bigcap_{i\in I}G_i\subseteq V$. Hence $\bigcap_{i\in I}{\ua}~\!\!g(G_i)\subseteq f^{-1}(V)\in\mathcal O(X)$. As $X$ is a strong $R$-space, there exists $J\in I^{(<\omega)}$ such that $\bigcap_{i\in J}\ua ~\!\!g(G_i)\subseteq f^{-1}(V)$. Then $g(\bigcap_{i\in J}G_i)\subseteq\bigcap_{i\in J}{\uparrow} ~\!\!(g(G_i))\subseteq f^{-1}(V)$, whence $\bigcap_{i\in J}G_i=f(g(\bigcap_{i\in J}G_i))\subseteq f(f^{-1}(V))\subseteq V$, and this completes the proof that $Y$ is a strong $R$-space.
\end{proof}

From Lemma \ref{prod-retract} and Proposition \ref{prop-retract-strongly-WF} we deduce the following.

\begin{proposition}\label{prop-prodct-strongly-WF}
	For a family $\{X_i :i\in I\}$ of $T_0$-spaces, if the product space $\prod_{i\in I}X_i$ is a strong $R$-space, then $X_i$ is a strong $R$-space for all $i\in I$.
\end{proposition}

In \cite[Example 5.5]{Xu-2026}, Xu gave two Scott spaces which are strongly well-filtered and $R$-spaces but their product space is not a strong $d$-space and hence it is neither a strongly well-filtered space nor an $R$-space. Now we further show that these two Scott spaces are strong $R$-spaces, and hence the property of being a strong $R$-space is not finite productive.

\begin{example}\label{exam-product-strong-$R$-space-is-not} By adding a top element $\infty$ to $\mathbb{N}^{+}$, we get a countable complete chain $\mathbb{N}^{+}_{\infty}=\mathbb{N}^{+}\cup \{\infty\}$. Let $Q=\{\omega_n : n\in\mathbb{N}^{+}\}\cup \{a, b\}$. Define an order on $Q$ as follows (see Figure 3):
\begin{enumerate}[\rm (1)]
\item $a<\omega_{n}$ and $b<\omega_n$ for all $n\in \mathbb{N}^{+}$,

\item for any $n, m\in \mathbb{N}^{+}$ with $n\neq m$, $\omega_n$ and $\omega_m$ are incomparable, and

\item $a$ and $b$ are incomparable.

\end{enumerate}

\begin{figure}[ht]
	\centering
	\includegraphics[height=5.0cm,width=12cm]{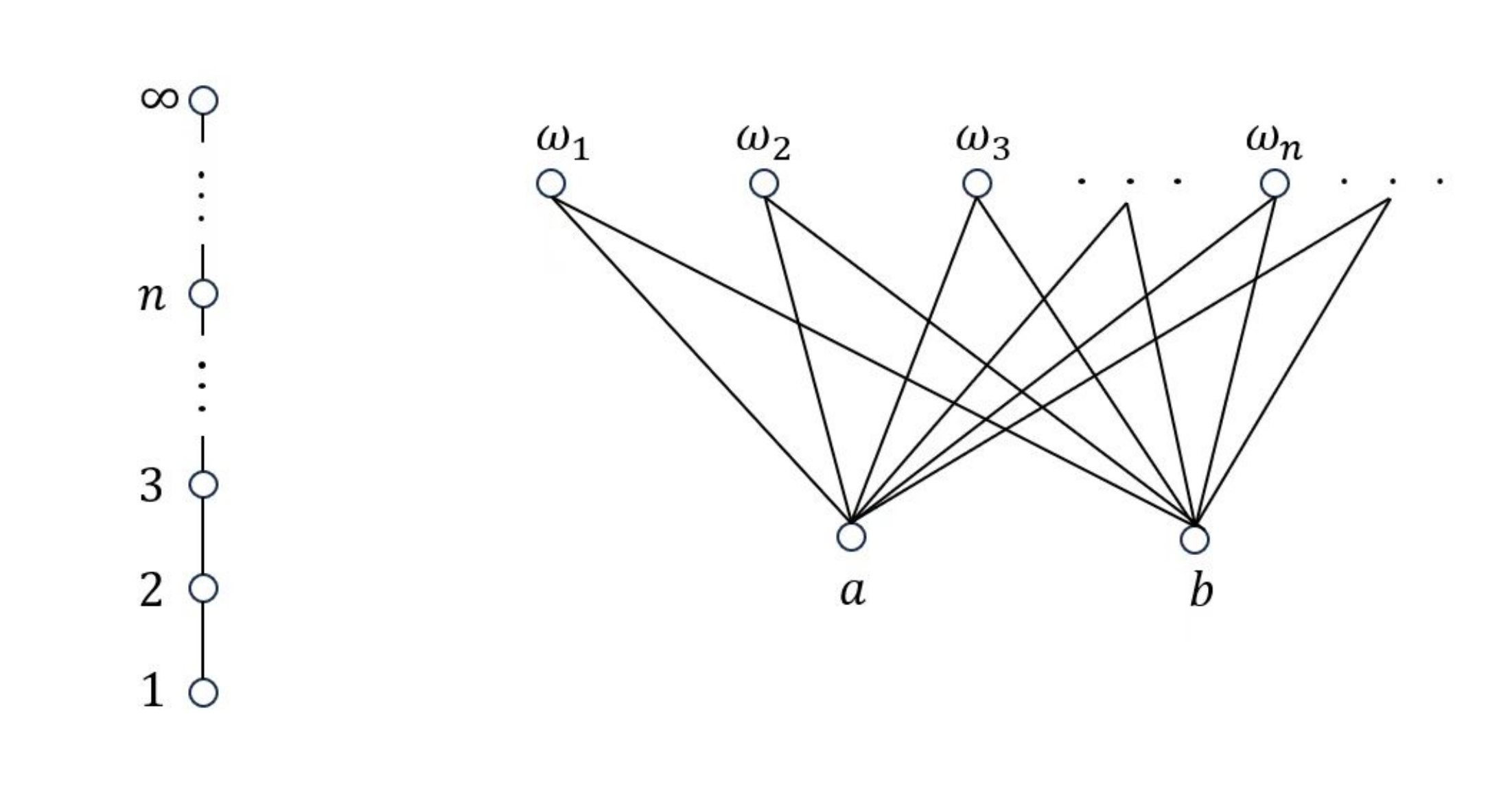}
	\caption{The countable complete chain $\mathbb{N}^{+}_\infty$ and countable Noetherian dcpo $Q$ in Example \ref{exam-product-strong-$R$-space-is-not}}
\end{figure}

\begin{figure}[ht]
	\centering
	\includegraphics[height=5.5cm,width=12cm]{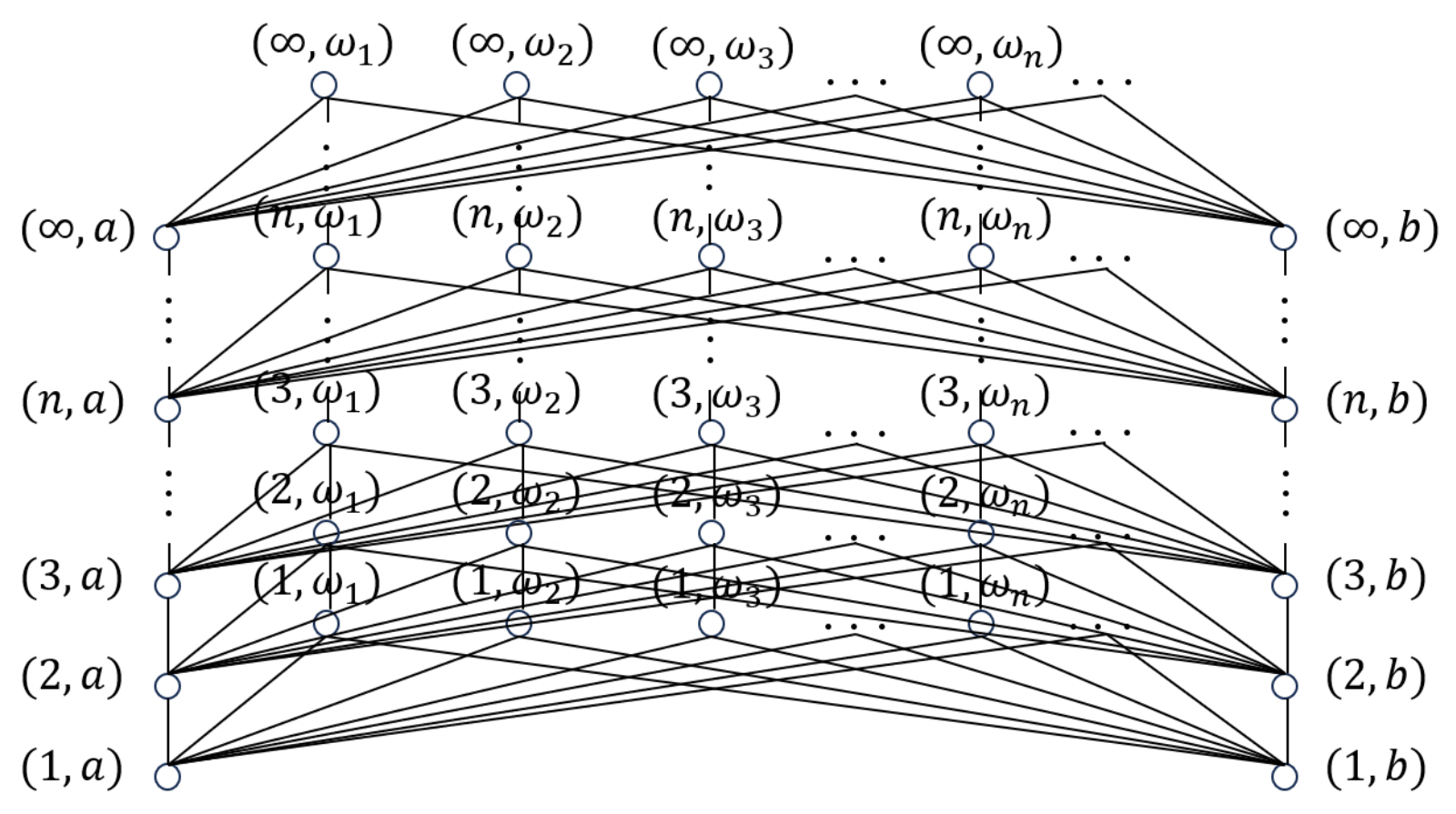}
	\caption{The product dcpo $\mathbb{N}^{+}_\infty\times Q$ in Example \ref{exam-product-strong-$R$-space-is-not}}
\end{figure}

Consider the product dcpo $\mathbb{N}^{+}_\infty \times Q$ (see Figure 4). Then we have the following conclusions:
\begin{enumerate}[\rm (a)]

\item $Q$ is a countable Noetherian dcpo.

\item $\sigma(Q)=\alpha(Q)$ and $Q(\Sigma~\!\!Q)=\mathbf{Fin}~\!Q$ by Lemma \ref{lem-compact-saturated-in-Alexanderoff-topologyin} and Proposition \ref{prop-Alexandroff-topology-sober}.

\item  $\Sigma~\!\!\mathbb{N}^{+}_\infty$ is a strong $R$-space by Proposition \ref{prop-complete-lattice-Scott-space-strong-$R$-space}.

    \item  $\Sigma~\!\!Q$ is a strong $R$-space.

    Let $\{{\uparrow}F_i : i\in I\}\subseteq Q(\Sigma~\!\!Q)=\mathbf{Fin}~\!Q$ and $U\in \sigma(Q)$ satisfying $\bigcap_{i\in I}{\ua}F_{i}\subseteq U$.

    {\bf Case 1:} $Q\setminus U$ is finite.

    By $\bigcap_{i\in I}{\ua}F_{i}\subseteq U$, we have $Q\setminus U\subseteq \bigcup_{i\in I} (Q\setminus {\ua}F_i)$. Hence there is $J\in I^{(<\omega)}$ such that $Q\setminus U\subseteq \bigcup_{i\in J}(Q\setminus {\ua}F_i)$, or equivalently, $\bigcap_{i\in J}{\ua}F_i\subseteq U$.

    {\bf Case 2:} $Q\setminus U$ is infinite.

    As $Q\setminus U$ is infinite, $\omega_m\in Q\setminus U$ for some $m\in \mathbb{N}^+$. Then there is $j\in I$ with $\omega_m\in Q\setminus {\ua}F_j$, whence $F_j$ is a finite subset of $\{\omega_n : n\in\mathbb{N}^+\}$. Therefore, $(Q\setminus U)\setminus (Q\setminus {\ua}F_j)\subseteq {\ua}F_j=F_j$ is finite, and consequently, there is $I_0\in I^{(<\omega)}$ such that $(Q\setminus U)\setminus (Q\setminus {\ua}F_j)\subseteq \bigcup_{i\in I_0}(Q\setminus {\ua}F_i)$. Hence $Q\setminus U\subseteq \bigcup_{i\in I_0\cup\{j\}}(Q\setminus {\ua}F_i)$, or equivalently, $\bigcap_{i\in I_0\cup\{j\}}{\ua}F_i\subseteq U$.

Therefore, $\Sigma~\!\!Q$ is a strong $R$-space.

\item $\Sigma~\!\!\mathbb{N}^{+}_\infty\times \Sigma~\!\! Q$ is not a strong $R$-space.

It was proved in \cite[Example 5.5]{Xu-2026} that $\Sigma~\!\!(\mathbb{N}^{+}_\infty\times Q)=\Sigma~\!\!\mathbb{N}^{+}_\infty\times \Sigma~\!\! Q$ and $\Sigma~\!\!(\mathbb{N}^{+}_\infty\times Q)$ is not a strong $d$-space. By Proposition \ref{prop-strongly-WF-is-WF}(2) and Proposition \ref{prop-strong-R-is-strong-WF}(1), $\Sigma~\!\!(\mathbb{N}^{+}_\infty\times Q)$ is not a strong $R$-space, and hence $\Sigma~\!\!\mathbb{N}^{+}_\infty\times \Sigma~\!\! Q$ is not a strong $R$-space.
\end{enumerate}
\end{example}

From Proposition \ref{prop-category-reflective-productive} and Example \ref{exam-product-strong-$R$-space-is-not} we immediately deduce the following result (comparing it with Theorem \ref{theor-sober-WF-d-space-reflective}).

\begin{theorem}\label{theor-category-strongly-WF-not-reflective} $\mathbf{S}$-$\mathbf{Top}_r$ is not reflective in $\mathbf{Top}_0$.
\end{theorem}

\section{Smyth power spaces, Scott power spaces and strong $R$-spaces}

The Smyth power space is a very important structure in domain theory, which plays a fundamental role in modeling the semantics of non-deterministic programming languages (see \cite{Abramsky-Jung-1994, GHKLMS-2003, Heckmann-1992, Schalk-1993}). There naturally arises a question of which topological properties are preserved by the Smyth power construction. In this paper, we focus on the property of being a strong $R$-space.

\begin{proposition}\label{prop-Smyth-power-space-$R$-space-X-is-also} For a $T_0$-space $X$, if $P_S(X)$ is an $R$-space, then $X$ is a strong $R$-space, and hence $X$ is an $R$-space.
\end{proposition}
\begin{proof} Assume $\{K_i : i\in I\}\subseteq Q(X)$ and $U\in \mathcal O(X)$ satisfying $\bigcap_{i\in I}K_i\subseteq U$. Then $\{\ua_{Q(X)}K_i : i\in I\}\subseteq Q((P_S(X)))$ and $\bigcap_{i\in I}\ua_{Q(X)}K_i=\{K\in Q(X) : K\subseteq \bigcap_{i\in I}K_i\}\subseteq \Box U\in \mathcal O(P_S(X))$. As $P_S(X)$ is an $R$-space, there is $J\in I^{(<\omega)}$ such that $\bigcap_{i\in J}{\ua_{Q(X)}}K_i\subseteq \Box U$. Then for any $x\in \bigcap_{i\in J}K_i$, we have ${\ua} x\in \{K\in Q(X) : G\subseteq \bigcap_{i\in J}K_i\}=\bigcap_{i\in J}{\ua_{Q(X)}}K_i\subseteq \Box U$, and hence $x\in U$. So $\bigcap_{i\in J}K_i\subseteq U$. Thus $X$ is a strong $R$-space.
\end{proof}

\begin{corollary}\label{cor-Smyth-power-space-strong-$R$-space-X-is-also} For a $T_0$-space $X$, if $P_S(X)$ is a strong $R$-space, then $X$ is a strong $R$-space.
\end{corollary}

\begin{remark}\label{rem-X-$R$-space-Smyth-not-$R$-space}
	Let $X$ be a countably infinite set and $X_{cof}$ the space equipped with the co-finite topology. Then $X_{cof}$ is an $R$-space but not a strong $R$-space (see Example \ref{exam-X-cof-T1-space-not-strong-$R$-space}). Therefore, $P_S(X_{coc})$ is not an $R$-space by Proposition \ref{prop-Smyth-power-space-$R$-space-X-is-also}. So the Smyth power construction does not preserve the property of being an $R$-space.
\end{remark}

However, we do not know whether the Smyth power construction preserves the property of being a strong $R$-space.

As pointed out by Goubault-Larreq in \cite{Goubault-2012}, there is another prominent topology one can put on $Q(X)$, namely, the Scott topology. The space $\Sigma Q(X)=(Q(X), \sigma (Q(X)))$ is naturally called the \emph{Scott power space} of $X$ and is shortly denoted by $P_{\sigma}(X)$.

The following example shows that there is a second-countable Noetherian $T_0$-space $X$ such that the Scott power space $P_{\sigma}(X)$ is a  coherent sober space (and hence it is a strong $R$-space) but $X$ is not well-filtered, whence $X$ is not a strong $R$-space.

\begin{example}\label{exm-Scott-power-space-strong-$R$-space-X-not-WF}
	Let $P=\mathbb{N}^+\cup\{\infty\}$ and define an order on $P$ by $x\leq_P y$ iff $x=y$ or $x\in\mathbb{N}^+$ and $y=\infty$ (see Figure 5). Then $P$ is a countable Noetherian dcpo. Let $\tau=\{(\mathbb{N}^+\setminus F)\cup \{\infty\} : F\in (\mathbb{N}^+)^{(<\omega)}\}\cup\{\emptyset, \{\infty\}, P\}$. It is straightforward to verify that $\tau$ is a $T_0$-topology on $P$ and the specialization order of $(P, \tau)$ agrees with the original order on $P$.

 \begin{figure}[ht]
	\centering
	\includegraphics[height=3cm,width=4cm]{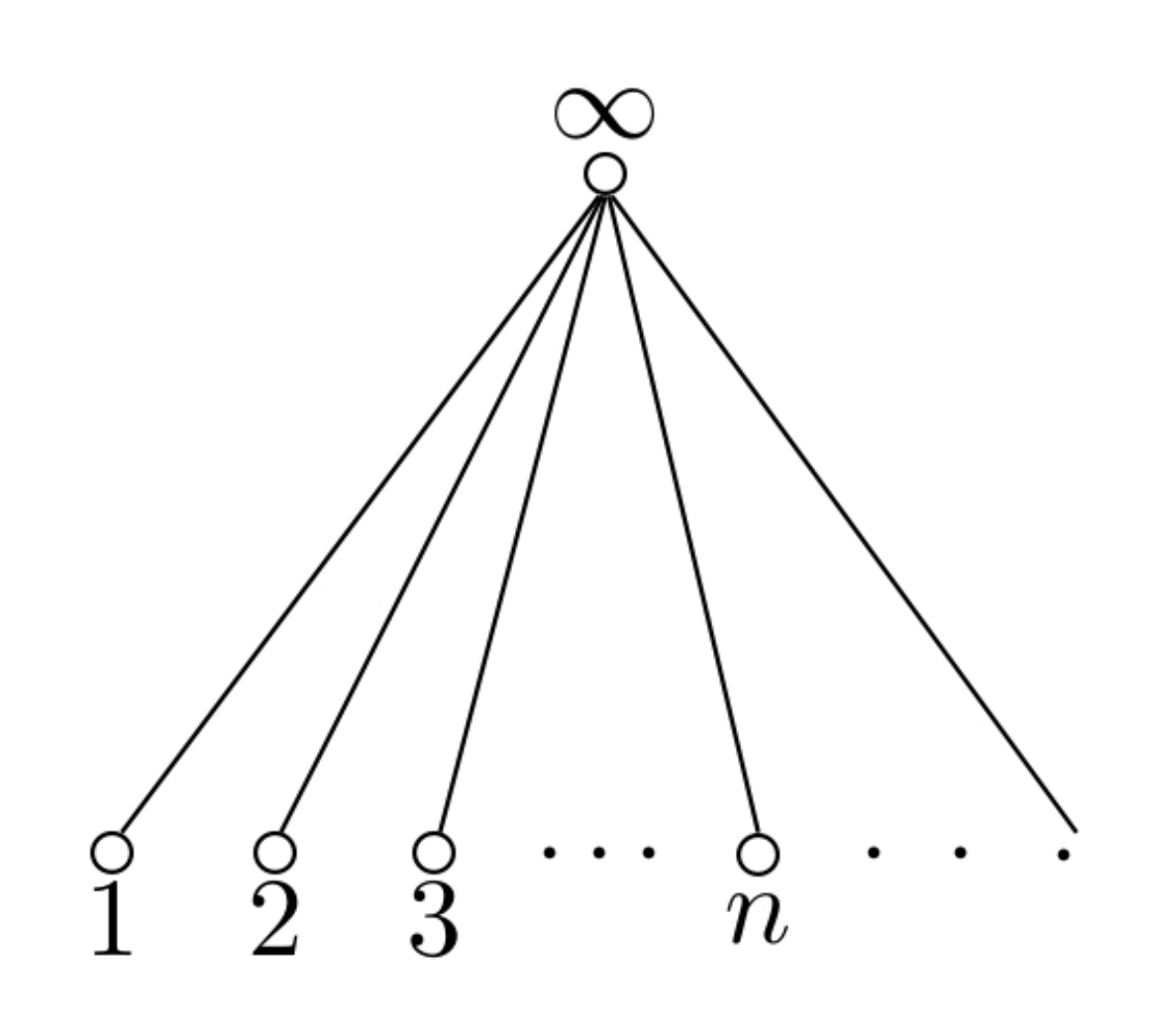}
	\caption{The countable Noetherian dcpo $P$ in Example \ref{exm-Scott-power-space-strong-$R$-space-X-not-WF}}
\end{figure}

Considering the $T_0$-space $(P, \tau)$, we have the following conclusions:
\begin{enumerate}[\rm (a)]
    \item $\upsilon(P) \subsetneqq \tau \subsetneqq \sigma(P)$.

Clearly, $\upsilon(P) \subseteq \tau$. As $\{\infty\}\in \tau$ and $\{\infty\}\notin \upsilon(P)$, we have $\upsilon(L) \subsetneqq \tau$. By Proposition \ref{prop-Alexandroff-topology-sober}, $\tau\subseteq \alpha(P)=\sigma(P)$. Since ${\uparrow}1=\{1, \infty\}\in \sigma(P)$, but $\{1, \infty\}\notin \tau$. Hence $\tau \subsetneqq \sigma(L)$.

    \item $\Gamma((P, \tau))=(\mathbb{N}^+)^{(<\omega)}\cup \{\emptyset, \mathbb{N}^+, P\}$ and $\ir_c((P, \tau))=\{\overline{\{n\}}=\{n\} : n\in \mathbb{N}^+\}\cup\{\overline{\{\infty\}}=P\}\cup\{\mathbb{N}^+\}$.
    \item $Q ((P, \tau))=\{A\cup \{\infty\} : A\subseteq \mathbb{N}^+\}$.
    \item $(P, \tau)$ is not well-filtered.

Let $\mathcal K=\{(\mathbb{N}^+\setminus F)\cup\{\infty\} : F\in (\mathbb{N}^+)^{(<\omega)}\}$. Then $\mathcal K\subseteq Q ((P, \tau))$ is a filtered family and $\bigcap \mathcal K=\{\infty\}\in \tau$. But $(\mathbb{N}^+\setminus F)\cup\{\infty\}=\{\infty\}$ for no $F\in (\mathbb{N}^+)^{(<\omega)}$. Thus $(P, \tau)$ is not well-filtered.

    \item $(P, \tau)$ is Noetherian and second-countable.

    Since $|\tau|=\omega$, $(P, \tau)$ is second-countable. As every subset of $P$ is compact in $(P, \tau)$, $(P, \tau)$ is a Noetherian space (and hence a locally compact space).

    \item $P_S((P, \tau))$ is second-countable.

Clearly, $\{\Box U : U\in \tau\}$ is a countable base of $P_S((P, \tau))$ (note that $|\tau|=\omega$). Hence $P_S((P, \tau))$ is second-countable.

\item $P_{\sigma}((P, \tau))$ is second-countable, coherent and sober.

    Clearly, $Q((P, \tau))$ is isomorphic with the algebraic lattice $2^{\mathbb{N}^+}$ (with the order of set inclusion) via the poset isomorphism $\varphi : Q ((P, \tau))\rightarrow 2^{\mathbb{N}^+}$ defined by $\varphi (A\cup \{\infty\})=\mathbb{N}^+\setminus A$ for each $A\in 2^{\mathbb{N}^+}$ (note that the order on $Q ((P, \tau))$ is the Smyth order). Hence $P_{\sigma}((P, \tau))=\Sigma~\!\!Q ((P, \tau))\cong \Sigma~\!\!2^{\mathbb{N}^+}$. As $2^{\mathbb{N}^+}$ is an algebraic lattice, by Propositions \ref{prop-algebraic-is-continuous}, \ref{prop-continuous-domain-Scott-is-sober} and \ref{prop-Scott-WF-coherent}, $\Sigma~\!\!2^{\mathbb{N}^+}$ is both sober and coherent, and hence $P_{\sigma}((P, \tau))$ is a coherent sober space. Clearly, $\Sigma~\!\!2^{\mathbb{N}^+}$ is second-countable since $\{\ua_{2^{\mathbb{N}^+}} F : F\in (2^{\mathbb{N}^+})^{(<\omega)}\}$ is a countable base of $\Sigma~\!\!2^{\mathbb{N}^+}$. So $P_{\sigma}((P, \tau))$ is second-countable.

\item $P_{\sigma} ((P, \tau))$ is a strong $R$-space by (g) and Proposition \ref{prop-strong-R-is-strong-WF}(2).
\end{enumerate}
\end{example}

But we do not know whether the Scott power space of a strong $R$-space is a strong $R$-space.

In the following, we investigate conditions under which the Smyth power space and Scott power space of a $T_0$-space is a strong $R$-space.
\begin{proposition}\label{prop-WF-Scott-WF} (\cite[Theorem 5.8]{Xu-Wen-Xi-2023}) For a well-filtered space $X$, the Scott power space $P_{\sigma}(X)$ is well-filtered.
\end{proposition}

\begin{proposition}\label{prop-coherent-WF-Scott-strong-$R$-space} For a coherent well-filtered space $X$, its Scott power space $P_{\sigma}(X)$ is a coherent well-filtered space. Therefore, $P_{\sigma}(X)$ is a strong $R$-space, and hence it is strongly well-filtered.
\end{proposition}
\begin{proof} By Proposition \ref{prop-WF-Scott-WF}, $P_{\sigma}(X)$ is a well-filtered space. Now we show that $P_{\sigma}(X)$ is coherent. For $K_1, K_2\in Q(X)$, by the coherence of $X$, $K_1\cap K_2\in Q(X)\cup \{\emptyset\}$. So ${\ua_{Q(X)}}K_1\bigcap {\ua_{Q(X)}}K_2={\ua_{Q(X)}}K_1\cap K_2\in Q{P_S(X)}\cup \{\emptyset\}$. Hence $P_{\sigma}(X)$ is coherent by Proposition \ref{prop-Scott-WF-coherent}, and consequently, $P_{\sigma}(X)$ is a strong $R$-space by Proposition \ref{prop-strong-R-is-strong-WF}(2). Therefore, $P_{\sigma}(X)$ is strongly well-filtered by Proposition \ref{prop-strong-R-is-strong-WF}(1).
\end{proof}

From Proposition \ref{prop-strongly-WF-is-WF}(4) and  Proposition \ref{prop-coherent-WF-Scott-strong-$R$-space} we deduce the following.

\begin{corollary}\label{cor-T2-Scott-strong-$R$-space} For a $T_2$-space $X$, its Scott power space $P_{\sigma}(X)$ is a coherent well-filtered space. Therefore, $P_{\sigma}(X)$ is a strong $R$-space, and hence it is strongly well-filtered.
\end{corollary}

\begin{lemma}\label{prop-V=S} (\cite[Theorem 5.7 and Corollaries 5.4, 5.6, 5.9-5.11]{Xu-Yang-2021}) Let $X$ be a $T_0$-space satisfying one of the following conditions:
\begin{enumerate}[\rm (1)]
\item $X$ is a core-compact (especially, locally compact) well-filtered space.
\item $X$ is a second-countable well-filtered space.
\item $X$ is a well-filtered space and $P_S(X)$ is first-countable.
 \item $X$ is a first-countable well-filtered space and $\mathrm{min}(K)$ is countable for any $K\in Q(X)$.
 \item $X$ is a first-countable well-filtered space in which all compact subsets are countable.
 \item $X$ is a countable, first-countable and well-filtered space.
 \item  $X$ is a metric space.
\end{enumerate}
\noindent Then the upper Vietoris topology and the Scott topology on $Q(X)$ coincide, that is, $P_S(X)=P_{\sigma}(X)$.
\end{lemma}

By Theorem \ref{theor-Schalk-Heckman-Keimel-theorem}, Proposition \ref{prop-CI-WF-sober}, Theorem \ref{theor-LC-sober=LC-wf=CC-sober} and Lemma \ref{prop-V=S}, we get the following result.

\begin{proposition}\label{prop-under-conditions-strong-$R$-space} Let $X$ be a coherent $T_0$-space satisfying one of the following conditions:
\begin{enumerate}[\rm (1)]
\item $X$ is a core-compact (especially, locally compact) well-filtered space.
\item $X$ is a second-countable well-filtered space.
\item $X$ is a well-filtered space and $P_S(X)$ is first-countable.
 \item $X$ is a first-countable well-filtered space and $\mathrm{min}(K)$ is countable for any $K\in Q(X)$.
 \item $X$ is a first-countable well-filtered space in which all compact subsets are countable.
 \item $X$ is a countable, first-countable and well-filtered space.
\end{enumerate}
\noindent Then $P_{S}(X)=P_{\sigma}(X)$, and $P_S(X)$ is both a sober space and a strong $R$-space. Hence it is strongly well-filtered.
\end{proposition}

The following corollary follows directly from Proposition \ref{prop-strong-R-is-strong-WF}(3) and Proposition \ref{prop-under-conditions-strong-$R$-space}.

\begin{corollary}\label{cor-T2-countable-strong-$R$-space} Let $X$ be a $T_2$-space satisfying one of the following conditions:
\begin{enumerate}[\rm (1)]
\item $X$ is a core-compact (especially, locally compact) space.
\item $X$ is second-countable.
\item $P_S(X)$ is first-countable.
 \item $X$ is a first-countable space and $\mathrm{min}(K)$ is countable for any $K\in Q(X)$.
 \item $X$ is a first-countable space in which all compact subsets are countable.
  \item $X$ is a countable first-countable space.
\end{enumerate}
\noindent Then $P_{S}(X)=P_{\sigma}(X)$, and $P_S(X)$ is both a sober space and a strong $R$-space. Hence it is strongly well-filtered.
\end{corollary}

\begin{lemma}\label{lem-metric-space-Smyth-power-space-first-countable} (\cite[Proposition 4.7]{Xu-Yang-2021}) For a metric space $(X, d)$, $P_S((X, d))$ is first-countable.
\end{lemma}

By Corollary \ref{cor-T2-countable-strong-$R$-space} and Lemma \ref{lem-metric-space-Smyth-power-space-first-countable}, we get the following.

\begin{corollary}\label{cor-metric-space-strong-$R$-space} For a metric space $(X, d)$, $P_{S}(X)=P_{\sigma}(X)$, and $P_S(X)$ is both a sober space and a strong $R$-space. Hence $P_{S}(X)$ is strongly well-filtered.
\end{corollary}

\section{Conclusions and further work}

In this paper, we have instigated the hereditary property of strong $R$-spaces and the closedness of strong $R$-spaces under retracts and finite products, and demonstrated some relations between the property of being a strong $R$-space and the de Groot dual topology. We have also given some conditions under which the Smyth power space and Scott power space of a $T_0$-space is a strong $R$-space, and shown that if the Smyth power space $P_S(X)$ of a $T_0$-space $X$ is an $R$-space, then $X$ is a strong $R$-space. Hence the Smyth power construction does not preserve the property of being an $R$-space. But we do not know whether the Smyth power construction and Scott power construction preserve the property of being a strong $R$-space.

We naturally pose the following two questions:

\begin{question}\label{ques-strong-$R$-space-Smyth-power-space-$R$-space} For a strong $R$-space $X$, is the Smyth power space $P_S(X)$ an $R$-space? Further, is $P_S(X)$ a strong $R$-space?
\end{question}

\begin{question}\label{ques-strong-$R$-space-Scott-power-space-$R$-space} For a strong $R$-space $X$, is the Scott power space $P_{\sigma}(X)$ a strong $R$-space?
\end{question}

The function spaces and categorical properties of strong $R$-spaces also deserve further investigation.

\end{document}